\documentclass[10pt]{article} % For LaTeX2e
\usepackage{nips14submit_e,times}
\usepackage{hyperref}
\usepackage{url}
\usepackage{amsmath,amsthm}
\usepackage{amsopn,amssymb}
\usepackage{bbm,dsfont}
\usepackage{mathrsfs}
\usepackage{enumerate}
\usepackage{shortcuts_MatComp}
\usepackage{graphicx}
\graphicspath{{images/}}

 \title{Probabilistic low-rank matrix completion on finite alphabets}

\author{Jean Lafond \\
 Institut Mines-T\'el\'ecom\\
T\'el\'ecom ParisTech\\ CNRS LTCI\\
\texttt{\scriptsize jean.lafond@telecom-paristech.fr}
\And
Olga Klopp\\
CREST et MODAL'X \\ Universit\'e Paris Ouest\\
\texttt{\scriptsize Olga.KLOPP@math.cnrs.fr}
\And
\'Eric Moulines \\
 Institut Mines-T\'el\'ecom\\
T\'el\'ecom ParisTech\\ CNRS LTCI\\
\texttt{\scriptsize moulines@telecom-paristech.fr}
\And
Joseph Salmon\\
 Institut Mines-T\'el\'ecom\\
T\'el\'ecom ParisTech\\ CNRS LTCI\\
\texttt{\scriptsize joseph.salmon@telecom-paristech.fr}
}

\nipsfinalcopy % Uncomment for camera-ready version

\begin{document}

\maketitle

\begin{abstract}
The task of reconstructing a matrix given a sample of observed
entries is known as the \emph{matrix completion problem}. It arises in
a wide range of problems, including recommender systems, collaborative
filtering, dimensionality reduction, image processing, quantum physics or multi-class classification
to name a few. Most works have focused on recovering an unknown real-valued low-rank
matrix from randomly sub-sampling its entries.
Here, we investigate the case where the observations take a finite number of values, 
corresponding for examples to ratings in 
recommender systems or labels in multi-class classification.
We also consider a general sampling scheme (not necessarily uniform) over the matrix entries.
The performance of a nuclear-norm penalized estimator is analyzed theoretically.
More precisely, we derive bounds for the Kullback-Leibler divergence between the true and estimated distributions.
In practice,  we have also proposed an efficient algorithm based on lifted coordinate gradient descent in order to tackle
potentially high dimensional settings.
\end{abstract}

\section{Introduction}
Matrix completion has attracted a lot of contributions over the past decade.
It consists in recovering the entries of a potentially high dimensional matrix,
based on their random and partial observations.
In the classical noisy matrix completion problem, 
the entries are assumed to be real valued and observed in presence of
additive (homoscedastic) noise.
In this paper, it is assumed that the entries take values in a finite alphabet that can model categorical data.
Such a problem arises in analysis of voting patterns, recovery of incomplete survey data
(typical survey responses are  true/false, yes/no or  do not know, agree/disagree/indifferent), 
quantum state tomography \cite{Gross11} (binary outcomes), 
% student responses data, 
recommender systems \cite{Koren_Bell_Volinsky09,Bobadilla_Ortega_Hernando_Gutierrez13} 
(for instance in common movie rating datasets,  \eg MovieLens or Neflix, ratings range from 1 to 5) 
among many others.
It is customary in this framework that rows represent individuals while columns represent items \eg movies,
survey responses, etc. 
Of course, the observations are typically incomplete, in the sense that a significant proportion of the entries are
missing. Then, a crucial question to be answered is whether it is possible to predict the missing entries from
these partial observations.

Since the problem of matrix  completion is ill-posed in general, it is necessary to impose a low-dimensional structure
on the matrix, one particularly popular example being a low rank constraint.
The classical noisy matrix completion problem (real valued observations and additive noise),
can be solved provided that the unknown matrix is low rank, either exactly or approximately;
see \cite{Candes_Plan10,Keshavan_Montanari_Oh10,Koltchinskii_Tsybakov_Lounici11,Negahban_Wainwright12,Cai_Zhou13,Klopp14} 
and the references therein.
Most commonly used methods amount to solve a least square program under a rank constraint or a
convex relaxation of a rank constraint provided by the nuclear (or trace norm) \cite{Fazel02}.

The problem of probabilistic low rank matrix completion over a finite alphabet has received much less attention; 
see \cite{Todeschini_Caron_Chavent13,Davenport_Plan_VandenBerg_Wootters12,Cai_Zhou14} among others.
To the best of our knowledge, only the binary case (also referred to as the 1-bit matrix completion problem)
has been covered in depth. In \cite{Davenport_Plan_VandenBerg_Wootters12}, 
the authors proposed to model the entries as Bernoulli random variables whose success rate depend upon
the matrix to be recovered through a convex link function (logistic and probit functions being natural examples).
The estimated matrix is then obtained as a solution of a maximization of the log-likelihood of the observations
under an explicit low-rank constraint. Moreover, the sampling model proposed in
\cite{Davenport_Plan_VandenBerg_Wootters12} assumes that the entries are
sampled uniformly at random. Unfortunately, this condition
is not totally realistic in recommender system applications: in such a context some users are
more active than others and some popular items are rated more frequently.
Theoretically, an important issue is that the method from
\cite{Davenport_Plan_VandenBerg_Wootters12} requires the knowledge of an upper bound on the
nuclear norm or on the rank of the unknown matrix.

Variations on the 1-bit matrix completion was further considered in \cite{Cai_Zhou14}
where a max-norm (though the name is similar, this is different from the sup-norm) 
constrained minimization is considered. The method of \cite{Cai_Zhou14}
allows more general non-uniform samplings but still
requires an upper bound on the max-norm of the unknown matrix.

In the present paper we consider a penalized maximum log-likelihood method, in which the log-likelihood of the observations 
is penalized by the nuclear norm (\ie we focus on the Lagrangian version rather than on the constrained one).
We first establish an upper bound of the Kullback-Leibler divergence between the
true and the estimated distribution under general sampling distributions; see \autoref{sec:main-results} for details. 
One should note that our method only requires the knowledge of 
an upper bound on the maximum absolute value of the probabilities,
and improves upon previous results found in the literature.

Last but not least, we propose an efficient implementation of our statistical procedure, 
which is adapted from the lifted coordinate descent algorithm recently introduced in 
\cite{Dudik_Harchaoui_Malick12,Harchaoui_Juditsky_Nemirovski14}. 
Unlike other methods, this iterative algorithm is designed to solve the convex optimization and not
(possibly non-convex) approximated formulation as in \cite{Recht_Re13}. It also has the benefit that it does not
need to perform full/partial SVD (Singular Value Decomposition) at every iteration; see \autoref{sec:numerical-experiment}
for details.

\subsection*{Notation}
Define $m_1\wedge m_2\eqdef\min(m_1,m_2)$ and  $m_1 \vee m_2\eqdef \max (m_1,m_2)$.
We equip the set of $m_1 \times m_2$ matrices with real entries (denoted $\RR^{m_1 \times m_2}$) with the scalar product
$\langle \mat{X}|\mat{X'} \rangle \eqdef \tr(\mat{X}^\top \mat{X'})$.
 For a given matrix $\mat{X}\in \RR^{m_1 \times m_2}$ we write
$\|\mat{X}\|_\infty \eqdef \max_{i,j} |\mat{X}_{i,j}|$ and, for $q \geq 1$, we denote its Schatten $q$-norm by
 \begin{equation*}
  \|\mat{X}\|_{\sigma,q}\eqdef \left( \sum_{i=1}^{m_1\wedge m_2} \sigma_i(\mat{X})^q \right)^{1/q} \eqs,
 \end{equation*}
where $\sigma_i(\mat{X})$ are the singular values of $\mat{X}$ ordered in decreasing order 
(see \cite{Bhatia97} for more details on such norms).
The operator norm of $\mat{X}$ is given by $\|\mat{X}\|_{\sigma,\infty}\eqdef\sigma_1(\mat{X})$.
%%%%%<Eric: j'ai viré les notations ci-dessous qui ne me semblent pas utiles>
%We denote by $\mathcal{S}_1(\mat{A}) \subset \RR^{m_1}$ (\resp $\mathcal{S}_2(\mat{A}) \subset \RR^{m_2}$)
%the linear spans generated by left (\resp right) singular vectors of $A$.
%$P_{\mathcal{S}^\bot_1(\mat{A})}$  (\resp $P_{\mathcal{S}^\bot_2(\mat{A})}$)
%denote the orthogonal projections on $\mathcal{S}^\bot_1(\mat{A})$ (\resp $\mathcal{S}^\bot_2(\mat{A})$).
%We then define the following orthogonal projections on $\RR^{m_1 \times m_2}$
%\begin{equation*}
%\Proj_{\mat{A}}^\bot:\mat{B} \to P_{\mathcal{S}^\bot_1(\mat{A})}\mat{B} P_{\mathcal{S}^\bot_2(\mat{A})}
%\text{ and } \Proj_{\mat{A}}\mat{B} \to \mat{B}-\Proj_{\mat{A}}^\bot(\mat{B})\eqs.
%\end{equation*}
%%%%%%
Consider two vectors of $\Class-1$ matrices 
$(\mat{X}^\class)_{\class=1}^{\Class-1}$ and $(\mat{X'}^\class)_{\class=1}^{\Class-1}$  such that 
for any $(k,l) \in [m_1]\times[m_2]$ we have
$\mat{X}^\class_{k,l}\geq 0$, $\mat{X'}^\class_{k,l}\geq 0$, $1-\sum_{\class=1}^{\Class-1}\mat{X}^\class_{k,l}\geq 0$
and $1-\sum_{\class=1}^{\Class-1}\mat{X'}^\class_{k,l}\geq 0$.
Their square Hellinger distance is
\begin{equation*}
\dhe^2(\mat{X},\mat{X'})\eqdef
\frac{1}{m_1m_2} \!\sum_{\substack{k \in [m_1] \\ l\in [m_2]}} \!\! \left[ \sum_{\class=1}^{\Class-1}\left( \sqrt{\mat{X}^\class_{k,l}}-\sqrt{\mat{X'}^\class_{k,l}} \right)^2
\!\!\!\!+ \!\! \left(\sqrt{1-\sum_{\class=1}^{\Class-1}\mat{X}^\class_{k,l}}-\sqrt{1-\sum_{\class=1}^{\Class-1}\mat{X'}^\class_{k,l}}\right)^2 \right]
\end{equation*}
and their Kullback-Leibler divergence is
\begin{equation*}
\KL{\mat{X}}{\mat{X'}}\eqdef\frac{1}{m_1m_2} \sum_{\substack{k \in [m_1] \\ l\in [m_2]}}
\left[ \sum_{\class=1}^{\Class-1} \mat{X}^\class_{k,l}\log{\frac{\mat{X}^\class_{k,l}}{\mat{X'}^\class_{k,l}} } + (1- \sum_{\class=1}^{\Class-1} \mat{X}^\class_{k,l})
\log{\frac{1- \sum_{\class=1}^{\Class-1} \mat{X}^\class_{k,l}}{1-\sum_{\class=1}^{\Class-1} \mat{X'}^\class_{k,l}} } \right] \eqs.
\end{equation*}

Given an integer $\Class>1$, a function $f:\RR^{\Class-1} \to \RR^{\Class-1}$ is called a $\Class$-link
function if for any $x\in \RR^{\Class-1}$ it satisfies  $f^\class(x) \geq 0$ 
for $\class\in [\Class-1]$ and $1-\sum_{\class=1}^{\Class-1} f^\class(x) \geq 0$.
For any collection of $\Class-1$ matrices $(\mat{X}^\class)_{\class=1}^{\Class-1}$, $f(\mat{X})$ denotes
the vector of matrices $(f(\mat{X})^\class)_{\class=1}^{\Class-1}$ such that $f(\mat{X})^\class_{k,l}=f(\mat{X}^\class_{k,l})$
for any $(k,l) \in [m_1] \times [m_2]$ and $\class \in [\Class-1]$.

\section{Main results}
\label{sec:main-results}
Let $\Class$ denote the cardinality of our finite alphabet, that is the number of classes of the logistic model
(\eg ratings have $\Class$  possible values or surveys $\Class$ possible answers). 
For a vector of $\Class-1$ matrices $X= (\mat{X^\class})_{\class=1}^{\Class-1}$
of $\RR^{m_1 \times m_2}$ and an index $\omega \in [m_1] \times [m_2]$, 
we denote by $X_\omega$ the vector $(X^\class_\omega)_{\class=1}^{\Class-1}$.
We consider an $\iid$ sequence $(\omega_i)_{1\leq i \leq n}$  over $[m_1] \times [m_2]$, with a probability distribution 
function $\Pi$
that controls the way the matrix entries are revealed. 
% We define $\pi_{k,l}\eqdef\Pi(\omega_1=(k,l))$ for $(k,l)\in [m_1]\times [m_2]$.
It is customary to consider the simple uniform sampling distribution over the set $[m_1] \times [m_2]$, 
though more general sampling schemes could be considered as well.
We observe $n$ independent random elements $(Y_i)_{1\leq i\leq n}\in [\Class]^n$.
The observations $(Y_1,\dots,Y_n)$ are assumed to be independent and to follow a multinomial distribution
with success probabilities given by
\begin{equation*}
  \PP(Y_i=\class)=f^\class(\mat{\tX^1}_{\omega_{i}},\dots ,\mat{\tX^{\Class-1}}_{\omega_{i}} ) \quad
  \class \in [\Class-1] \quad \text{and} \quad \PP(Y_i=\Class)=1-\sum_{\class=1}^{\Class-1} \PP(Y_i=\class)
\end{equation*}
where $\{ f^\class \}_{\class=1}^{\Class-1}$ is a $p$-link function 
%taking value $\{(a^1,\dots,a^{\Class}), a^\class \geq 0, \class = 1,\dots, \Class, \sum_{\class=1}^{\Class} a^\class = 1\}$; 
and $\tX = (\tX^\class)_{\class=1}^{\Class-1}$  
is the vector of true (unknown) parameters we aim at recovering.
For ease of notation, we often write $\mat{\tX}_i$ instead of  $\mat{\tX}_{\omega_i}$.
Let us denote by $\Lik$ the (normalized) negative log-likelihood of the observations:
\begin{equation}
\label{eq:likelihood-multinomial}
   \Lik(\mat{X}) = -\frac{1}{n} \sum_{i=1}^{n} \left[
   \sum_{\class=1}^{\Class-1} \1_{\{Y_i=\class\}}\log\left(f^\class(\mat{X}_i)\right) 
   + \1_{\{Y_i=\Class\}} \log\left(1-\sum_{\class=1}^{\Class-1}f^\class(\mat{X}_i) \right) \right] \eqs,  
\end{equation}
%We denote by $\mathscr{E}$, the set of canonical matrices $\mathscr{E}\eqdef\{E_{k,l} : \: (k,l)\in [m_1] \times [m_2] \}$.
%%$E_{k,l}=e_k e_l^\top $ with
%%$e_k \eqdef (\delta_{k}(k'))_{1 \leq k' \leq m_1}$,  and $e_l \eqdef (\delta_{l}(l'))_{1 \leq l' \leq m_2}$,
%With these notation, $\mat{X}_{\omega_i} = \pscal{\mat{X}}{\mat{E}_{\omega_i}}$ and the log-likelihood can be written as,
%\begin{equation*}
% \Lik(\mat{X}) = \sum_{i=1}^{n}\left[ \1_{\{Y_i=1\}}\log\left(f(\pscal{\mat{X}}{E_{\omega_i}})\right) + \1_{\{Y_i=-1\}}\log\left(1-f(\pscal{\mat{X}}{E_{\omega_i}})\right)\right]\eqs.
%\end{equation*}
For any $\gamma>0$ our proposed estimator is the following:
\begin{equation}
\label{eq:MinPb}
 \hat{\mat{X}}=\argmin_{\substack{\mat{X} \in (\RR^{m_1 \times m_2})^{\Class-1} \\
 \max_{\class \in [\Class-1]} \|\mat{X}^\class \|_{\infty}\leq \gamma }} \Obj(\mat{X}) \eqsp, \quad \text{where} \quad  \Obj(\mat{X})= \Lik(\mat{X}) +
 \lambda \sum \limits_{\class=1}^{\Class-1} \|\mat{X}^\class \|_{\sigma,1} \eqsp,
\end{equation}
with $\lambda>0$ being a regularization parameter controlling the rank of the estimator. 
In the rest of the paper we assume that the negative log-likelihood $\Lik$ is convex (this is the case for the multinomial logit function,
see for instance \cite{Boyd_Vandenberghe04}).

% The estimator we study is defined as follows:
% \begin{equation}
% \label{MinPb}
%  \hat{\mat{X}}=\argmin_{\substack{\mat{X} \in \RR^{m_1 \times m_2}, \|\mat{X}\|_{\infty}\leq \gamma }} \Obj(\mat{X}) \eqs,
% \end{equation}
% where
% \begin{equation*}
%  \Obj(\mat{X})= -\frac{1}{n}\Lik(\mat{X}) + \lambda \|\mat{X}\|_{\sigma,1}\eqs.
% \end{equation*}
% with $\lambda>0$ a regularization parameter.

In this section we present two results controlling the estimation error of $\hat{\mat{X}}$
in the binomial setting (\ie\ when $\Class=2$).
Before doing so, let us introduce some additional notation and assumptions.
The score function (defined as the gradient of the negative log-likelihood)
taken at the true parameter $\tX$, is denoted by $\bar{\Sigma}\eqdef \nabla \Lik(\mat{\tX})$.
% $\Sigma_Y(\mat{X}) \eqdef \nabla \Lik(\mat{X})\eqsp$.
% For the true parameter matrix $\mat{\tX}$ we define similarly $\bar{\Sigma}\eqdef\Sigma_Y(\mat{\tX})$.
We also need the following constants depending on the link function $f$ and $\gamma > 0$:
\begin{align*}
M_\gamma= &\sup_{|x|\leq\gamma}2 |\log(f(x))| \eqsp,\\
L_\gamma= &\max\left(\sup_{|x|\leq\gamma} \frac{|f'(x)|}{f(x)},\sup_{|x|\leq\gamma} \frac{|f'(x)|}{1-f(x)}\right) \eqsp, \\
 K_\gamma= &\inf_{|x|\leq \gamma} \frac{f'(x)^2}{8f(x)(1-f(x))} \eqsp.
\end{align*}

In our framework, we allow for a general distribution for observing the coefficients.
However, we need to control deviations of the sampling mechanism from the uniform distribution
and therefore we consider the following assumptions.
\begin{assumption}
\label{A1}
There exists a constant $\mu \geq 1$ such that for all
indexes $(k,l) \in [m_1] \times [m_2]$
\begin{equation*}
\min_{k,l} (\pi_{k,l}) \geq 1/(\mu m_1 m_2)\eqsp.
\end{equation*}
with  $\pi_{k,l} \eqdef \Pi(\omega_1 = (k,l))$.
\end{assumption}
Let us define $C_l\eqdef\sum_{k=1}^{m_1} \pi_{k,l}$ (resp.
$R_k \eqdef \sum_{l=1}^{m_2} \pi_{k,l}$) for any $l\in[m_2]$ (resp. $k \in [m_1]$)
the probability of sampling a coefficient in column $l$ (resp. in row $k$).
\begin{assumption}
\label{A2}
There exists a constant $\Lc \geq 1$ such that
\begin{equation*}
 \max_{k,l}(R_k,C_l) \leq \Lc/(m_1 \wedge m_2) \eqsp,
\end{equation*}

\end{assumption}
Assumption \autoref{A1} ensures that each coefficient has a non-zero probability
of being sampled whereas \autoref{A2} requires that no column nor row is sampled with too high probability 
(see also \cite{Foygel_Salakhutdinov_Shamir_Srebro11,Klopp14}
for more details on this condition).

We define the sequence of matrices $(E_i)_{i=1}^{n}$ associated to the revealed coefficient $(\omega_i)_{i=1}^{n}$ by
$E_i\eqdef~e_{k_i} (e'_{l_i})^\top$ where $(k_i,l_i)=\omega_i$ and with
$(e_k)_{k=1}^{m_1}$ (\resp\ ($e'_l)_{l=1}^{ m_2}$) being the canonical basis of
$\RR^{m_1}$ (\resp\ $\RR^{m_2}$).
Furthermore, if $(\varepsilon_i)_{1\leq i \leq n}$ is a Rademacher sequence independent 
 from  $(\omega_i)_{i=1}^{n}$ and $(Y_i)_{1\leq i \leq n}$ we define
\begin{equation*}
 \Sigma_R\eqdef\frac{1}{n}\sum_{i=1}^{n} \varepsilon_i E_i \eqsp.
\end{equation*}
We can now state our first result. For completeness, the proofs can be found in the supplementary material.
\begin{theorem}
\label{th1}
 Assume \autoref{A1} holds,   $\lambda \geq~2\| \bar{\Sigma} \|_{\sigma,\infty}$ and $\| \tX \|_\infty \leq \gamma$.
 Then, with probability at least 
 $
 1-2/d
 $
 the Kullback-Leibler divergence between the true and estimated distribution 
is bounded by 
 \begin{equation*}
  \KL{f(\mat{\tX})}{f(\hat{\mat{X}})} \leq 8\max \left( \frac{\mu^2}{K_\gamma}m_1 m_2\rank(\mat{\tX}) 
  \left( \lambda^2 +  c^* L_\gamma^2(\PE \|\Sigma_R \|_{\sigma,\infty})^2\right) , \mu eM_\gamma\frac{\sqrt{\log(d)}}{n} \right),
 \end{equation*}
where
$ c^*$ is a universal constant.
\end{theorem}
% \begin{proof}
%  The proof is postponed to \autoref{proofth1}.
% \end{proof}

Note that $\| \bar{\Sigma} \|_{\sigma,\infty}$ is stochastic and that its expectation $\EE \|\Sigma_R \|_{\sigma,\infty}$
is unknown. However, thanks to Assumption \autoref{A2} these quantities can be controlled. 

% \begin{assumption}
% \label{A2}
% There exists a constant $\Lc >0$ such that
% \begin{equation*}
%  \max_{k,l}(R_k,C_l) \leq \frac{\Lc}{m_1 \wedge m_2} \eqsp,
% \end{equation*}
% \end{assumption}
% The last assumption requires that no column nor row is sampled with high probability 
% (see also \cite{Foygel_Salakhutdinov_Shamir_Srebro11,Klopp14})
% for more details on this condition).
To ease notation let us also define $m\eqdef m_1 \wedge m_2 $, $M\eqdef m_1 \vee m_2 $  and $d\eqdef m_1+m_2$. 

\begin{theorem}
\label{th2}
Assume \autoref{A1} and \autoref{A2} hold and that $\| \tX \|_\infty \leq \gamma$. Assume in addition that $n \geq  2 m \log(d)/(9\Lc)$.
Taking $\lambda = 6L_\gamma\sqrt{ 2 \Lc \log(d)/(mn)}$, then with probability at least  $1-3/d$
the folllowing holds
  \begin{equation*}
  K_\gamma\frac{\|\tX-\hat{X}\|_{\sigma,2}^2}{m_1 m_2 }\leq
  \KL{f(\mat{\tX})}{f(\hat{\mat{X}})} \leq \max \left( \bar{c}\frac{\nu\mu^2L_\gamma^2}{K_\gamma}\frac{ M\rank(\mat{\tX}) \log(d)}{n} ,
  8\mu eM_\gamma\frac{\sqrt{\log(d)}}{n} \right),
 \end{equation*}
where $\bar{c}$ is a universal constant.
\end{theorem}

\begin{remark*}
Let us compare the rate of convergence of \autoref{th2} with those obtained in previous works on 1-bit matrix completion.
 In \cite{Davenport_Plan_VandenBerg_Wootters12}, the parameter
  $\bar X$ is estimated by minimizing the negative log-likelihood
  under the constraints $\|\mat{X}\|_\infty \leq \gamma$ and $\|\mat{X}\|_{\sigma,1} \leq  \gamma
 \sqrt{rm_1m_2}$ for some $r>0$.
  Under the assumption that $\rank(\mat{\tX}) \leq r$, they could prove that
  \begin{equation*}
   \frac{\|\mat{\tX}-\hat{\mat{X}}\|^2_{\sigma,2}}{m_1m_2} \leq C_\gamma \sqrt{\frac{rd}{n}} \eqsp,
  \end{equation*}
 where $C_\gamma$ is a constant depending on $\gamma$ (see \cite[Theorem 1]{Davenport_Plan_VandenBerg_Wootters12}).
 This rate of convergence is slower than the rate of convergence given by \autoref{th2}. \cite{Cai_Zhou14} studied a max-norm
 constrained maximum likelihood estimate and obtained a rate of convergence similar to \cite{Davenport_Plan_VandenBerg_Wootters12}.
\end{remark*}

\section{Numerical Experiments}
\label{sec:numerical-experiment}
\textbf{Implementation} 
\quad For numerical experiments, data were simulated according to a multinomial logit distribution.
In this setting, an observation $Y_{k,l}$
associated to row $k$ and column $l$ is distributed as
$\PP(Y_{k,l}=\class)=f^\class( \mat{X}^1_{k,l},\dots ,\mat{X}^{\Class-1}_{k,l})$ where
\begin{equation}\label{eq:logitlink}
f^\class(x_1,\dots,x_{\Class-1}) =\exp(x_\class )\left(1+\sum \limits_{\class =1}^{\Class-1} \exp(x_\class )\right)^{-1},
\quad \text{ for } \class \in [\Class-1] \eqsp.
\end{equation}
With this choice, $\Lik$ is convex and problem \eqref{eq:MinPb} can be solved using convex optimization
algorithms. Moreover, following the advice of \cite{Davenport_Plan_VandenBerg_Wootters12} we
considered the unconstrained version of problem \eqref{eq:MinPb} (\ie with no constraint on $\|X\|_\infty$),
which reduces significantly the computation burden and has no significant impact on the solution in practice.
To solve this problem,
we have extended to the multinomial case the coordinate gradient
descent algorithm introduced by \cite{Dudik_Harchaoui_Malick12}.
This type of algorithm has the advantage, say over the Soft-Impute \cite{Mazumder_Hastie_Tibshirani10}
or the SVT \cite{Cai_Candes_Shen10}
algorithm, that it does not require the computation of a full SVD at each step
of the main loop of an iterative (proximal) algorithm (bare in mind that the proximal operator
associated to the nuclear norm is the soft-thresholding operator of the singular values).
The proposed version only computes the largest singular vectors and singular values.
This potentially decreases the computation by a factor close to the value of the upper bound on the rank commonly used
(see the aforementioned paper for more details).

Let us present the algorithm.
Any vector of $p-1$ matrices $\mat{X}=(\mat{X^\class})_{\class=1}^{\Class-1}$ 
is identified as an element of the tensor product space
$\RR ^{m_1 \times m_2}~\otimes~\RR ^{\Class-1}$ and denoted by:
\begin{equation}
\label{eq:rep_tens}
\mat{X}=\sum \limits_{\class=1}^{\Class-1} \mat{X}^\class \otimes e^\class \eqs,
\end{equation}
where again $(e^\class)_{\class=1}^{\Class-1}$ is the canonical basis on $\RR ^{\Class-1}$
and $\otimes$ stands for the tensor product.
The set of normalized rank-one matrices is denoted by
\begin{equation*}
 \tens:=\left\{ M \in \RR ^{m_1 \times m_2} |
 M=  uv^\top ~|~ \|u \|=\|v \|=1, u \in \RR^{m_1},v \in \RR^{m_2} \right\} \,.
\end{equation*}
Define  $\Theta$ the linear space of real-valued functions on $\tens$ with finite support, \ie\  $\theta(M)= 0$ except for a finite number of $M \in \tens$.
This space is equipped with the $\ell^1$-norm  $\| \theta \|_1= \sum_{M \in \tens} |\theta(M)|$.
Define by $\Theta_+$ the positive orthant, \ie\ the cone of functions $\theta \in \Theta$ such that $\theta(M) \geq 0$ for all $M \in \tens$.
Any tensor $\mat{X}$ can be associated with a vector $\theta=(\theta^1,\dots,\theta^{\Class-1}) \in \Theta^{\Class-1}_+$, \ie\
\begin{equation}
\label{eq:representation-canonique}
 \mat{X}= \sum_{\class=1}^{\Class-1} \sum_{M \in \tens} \theta^\class(M) M \otimes e^\class \eqs .
\end{equation}
Such representations are not unique, and among them, the one associated to
the SVD plays a key role, as we will see below.
For a given $X$ represented by \eqref{eq:rep_tens} and for any $\class \in \{1,\dots, \Class-1\}$, 
denote by $\{\sigma_k^\class \}_{k=1}^{n^\class}$
the (non-zero) singular values of the matrix $X^\class$
and $\{u_k^\class$,$v_k^\class\}_{k=1}^{n^\class}$ the associated singular vectors. Then,
$X$ may be expressed as
\begin{equation}
 \label{eq:decomposition-svd}
 X= \sum_{\class=1}^{\Class-1} \sum_{k=1}^{n^\class} \sigma_k^\class u_k^\class (v_k^\class)^{\top} \otimes e^\class \, .
\end{equation}
Defining $\theta^\class$ the function $\theta^\class(M)= \sigma_k^\class$ if
$M= u_k^\class (v_k^\class)^{\top}$, $k\in[ n^\class]$ and $\theta^\class(M)= 0$ otherwise,
one obtains a representation of the type given in Eq.~\eqref{eq:representation-canonique}.

Conversely, for any  $\theta=(\theta^1,\dots,\theta^{\Class-1}) \in \Theta^{\Class-1}$,  define the map
\begin{equation*}
 \wrec: \theta \to \wrec_\theta 
 :=\sum_{\class=1}^{\Class-1} W^j_\theta \otimes e^\class \quad \text{with}
 \quad W^j_\theta\eqdef\sum_{M \in \tens} \theta^\class(M) M  
\end{equation*}
and the auxiliary objective function
\begin{equation}
\label{eq:definition-Objb}
\Objb(\theta)= \lambda \sum_{\class=1}^{\Class-1} \sum_{M \in \tens} \theta^\class(M) + \Lik(\wrec_\theta) \eqs.
\end{equation}
The map $\theta \mapsto \wrec_\theta$ is a continuous linear map
from $(\Theta^{\Class-1}, \|\cdot\|_1)$ to $\RR^{m_1 \times m_2}~\otimes~\RR^{\Class-1}$, where
$\| \theta \|_1= \sum_{\class=1}^{\Class-1} \sum_{M \in \tens} |\theta^\class(M)|$.
In addition, for all $\theta \in \Theta^{\Class-1}_+$
\[
\sum_{\class=1}^{\Class-1} \|\mat{\wrec_\theta}^\class \|_{\sigma,1} \leq \| \theta \|_1 \eqs,
\]
and one obtains
$\|\theta\|_1= \sum_{\class=1}^{\Class-1} \| \mat{\wrec_\theta^\class} \|_{\sigma,1}$
when $\theta$ is the representation associated to the SVD decomposition.
An important consequence, outlined in \cite[Proposition~3.1]{Dudik_Harchaoui_Malick12},
is that the minimization of \eqref{eq:definition-Objb} 
is actually equivalent to the minimization of  \eqref{eq:MinPb}; see \cite[Theorem~3.2]{Dudik_Harchaoui_Malick12}.

The proposed coordinate gradient descent algorithm updates at each step the nonnegative finite support function $\theta$.
For $\theta \in \Theta$ we denote by $\supp(\theta)$ the support of $\theta$ and for $M \in \tens$,
by $\delta_M \in \Theta$ the Dirac function on $\tens$ satisfying $\delta_M(M)=1$ and $\delta_M(M')=0$ if $M' \neq M$.
In our experiments we have set to zero the initial $\theta_0$.

\begin{algorithm}[h]
% \dontprintsemicolon
\KwData{Observations: $Y$, tuning parameter $\lambda$\\
initial parameter: $\theta_0 \in  \Theta^{\Class-1}_+$;  tolerance: $\epsilon$;  maximum number of iterations: $K$ }
\KwResult{$\theta \in  \Theta^{\Class-1}_+$}
\textbf{Initialization:} $\theta \leftarrow \theta_0$, $k \leftarrow 0$\\
\While{$k\leq K$ \rm}{
\For{$\class=0$ to $\Class-1$}{
Compute top singular vectors pair of $\left(-\nabla \Lik(\mat{W_\theta})\right)_\class$: $u_\class$, $v_\class$
}
Let $g=\lambda +\min_{\class={ 1,\dots, \Class-1}} \pscal{\nabla\Lik}{u^\class(v^\class)^\top}$\\
\If{$g \leq -\epsilon/2$}{
  $(\beta_0,\dots,\beta_{\Class-1})=\displaystyle \argmin_{(b_0,\dots,b_{\Class-1}) 
  \in \RR_+^{\Class-1}} 
  \Objb\left(\theta+ (b_0 \delta_{u^0(v^0)^\top},\dots,b_{\Class-1} \delta_{u^{\Class-1}(v^{\Class-1})^\top})\right)$\\
  $\theta \leftarrow \theta+ (\beta_0 \delta_{u^0(v^0)^\top},\dots,\beta_{\Class-1} \delta_{u^{\Class-1}(v^{\Class-1})^\top})$\\
  $k\leftarrow k+1$
}
\Else{
  Let $g_{\rm max}=  \max_{\class\in [\Class-1]} \max_{u^\class(v^\class)^\top\in \supp(\theta^\class)}
  |\lambda + \pscal{\nabla\Lik}{u^\class(v^\class)^\top}|$ \\
  \If{$g_{\rm max} \leq \epsilon$}{
  \textbf{break}
  }\Else{
  $\theta \leftarrow \displaystyle \argmin_{\theta' \in \Theta_+^{\Class-1},
  \supp(\theta'^\class) \subset \supp(\theta^\class), \class \in [\Class-1]} \Objb(\theta')$\\
  $k\leftarrow k+1$
  }
}
}
\caption{Multinomial lifted coordinate gradient descent}
\label{alg:MCLGD}
\end{algorithm}
A major interest of Algorithm \ref{alg:MCLGD} is that it requires to store the value of the parameter
entries only for the indexes which are actually observed. Since in practice
the number of observations is much smaller than the total number of coefficients $m_1 m_2$, 
this algorithm is both memory and computationally efficient.
Moreover, using an SVD algorithm such as Arnoldi
iterations to compute the top singular values and vector
pairs (see \cite[Section 10.5]{Golub_VanLoan13} for instance)
allows us to take full advantage of gradient sparse structure.
Algorithm \ref{alg:MCLGD} was implemented in C and Table \ref{tab:exec} gives a rough idea of the execution 
time for the case
of two classes on a 3.07Ghz w3550 Xeon CPU (RAM 1.66 Go, Cache 8Mo).

\begin{table*}[t!]
\centering
\begin{tabular}{|l|c|c|c|}
   \hline
   \textbf{Parameter Size} & $10^3\times10^3$ &$3\cdot10^3\times3\cdot10^3$ &$10^4\times10^4$   \\
   \hline
   \textbf{Observations} & $10^5$ &$10^5$ & $ 10^7$   \\
   \hline
   \textbf{Execution Time (s.)}& $4.5$ & $52$ & $730$ \\
   \hline
\end{tabular}
\caption{Execution time of the proposed algorithm for the binary case.
}
\label{tab:exec}
\end{table*}

\textbf{Simulated experiments}
\quad To evaluate our procedure we have performed simulations for matrices with $\Class=2$ or $5$.
For each class matrix $\mat{X^j}$ we sampled uniformly five unitary vector pairs $(u_k^\class, v_k^\class)_{k=1}^{5}$. We have then generated matrices of rank equals to 5,
such that
% each class $\class \in \{1,\dots,\Class-1\}$ and using
\begin{equation*}
 \mat{X}^j=\Gamma\sqrt{m_1m_2} \sum \limits_{k=1}^{5}  \alpha_k u_k^\class( v_k^\class)^\top \eqs,
\end{equation*}
with $(\alpha_1,\dots,\alpha_5)=(2,1,0.5,0.25,0.1)$ and $\Gamma$ is a scaling factor. The $\sqrt{m_1m_2}$
factor,
guarantees that $\EE[\|X^\class\|_{\infty}]$ does not depend on the sizes of the problem $m_1$ and $m_2$.

We then sampled the entries uniformly and the 
observations according to a logit distribution given by Eq.~\eqref{eq:logitlink}.
We have then considered and compared the two following estimators both computed using Algorithm~\ref{alg:MCLGD}:
\begin{itemize}
 \item the logit version of our method (with the link function given by Eq.~\eqref{eq:logitlink})
 \item the Gaussian completion method (denoted by $\hat{X}^\mathcal{N}$), that consists in using the Gaussian log-likelihood instead of
the multinomial in
\eqref{eq:MinPb}, \ie using a classical squared Frobenius norm (the implementation being adapted mutatis mutandis). Moreover 
an estimation of the standard deviation is obtained by the classical analysis of the residue.
\end{itemize}

Contrary to the logit version, the Gaussian matrix completion does not directly recover the probabilities of observing a rating.
However, we can estimate this probability by the following quantity:
\begin{equation*}
{\PP}(\mat{ \hat{X}}_{k,l}^\mathcal{N}=\class)=F_{\mathcal{N}(0,1)}(p_{\class+1})-F_{\mathcal{N}(0,1)}(p_\class)
 \text{  with  } p_\class =
 \begin{cases}
 0 &\text{ if } \class =1\eqs,\\
\frac{\class-0.5-\hat{\mat{X}}^\mathcal{N}_{k,l}}{\hat{\sigma}} &\text{ if } 0<\class<\Class \\
1  &\text{ if }\class = \Class \eqs,
\end{cases}
\end{equation*}
where $F_{\mathcal{N}(0,1)}$ is the cdf of a zero-mean standard Gaussian random variable.

% This allows us to compare the KL-error between the Gaussian and the logistic estimator.
As we see on \autoref{fig:KL},
the logistic estimator outperforms the Gaussian for both cases $\Class=2$ and $\Class=5$ in terms of the Kullback-Leibler
divergence.
This was expected because the Gaussian
model allows uniquely symmetric distributions with the same variance for all the ratings,
which is not the case for logistic distributions. The choice of the $\lambda$ parameter has been set for
both methods by performing 5-fold cross-validation on a geometric grid of size $0.8 \log(n)$.

\autoref{tab:sim_bin_error} and \autoref{tab:sim_mult_error} summarize
the results obtained for a $900 \times 1350$ matrix respectively for $\Class=2$ and  $\Class=5$.
For both the binomial case $\Class=2$ and  the multinomial case $\Class=5$,
the logistic model slightly outperforms the Gaussian model.
This is partly due to the fact that in the multinomial case, some
ratings can have a multi-modal distribution.
In such a case, the Gaussian model is
unable to predict these ratings, because its distribution is necessarily centered around a single value and is not flexible
enough. For instance consider the case of a rating distribution with high probability of seeing 1 or 5,
low probability of getting 2, 3 and 4, where we observed both $1$'s and $5$'s.
The estimator based on a Gaussian model will tend to center its distribution around 2.5 and therefore misses the bimodal shape of the distribution.

\begin{figure}[ht!]
\centering
\begin{minipage}{0.47\textwidth}
 \includegraphics[width=1\textwidth]{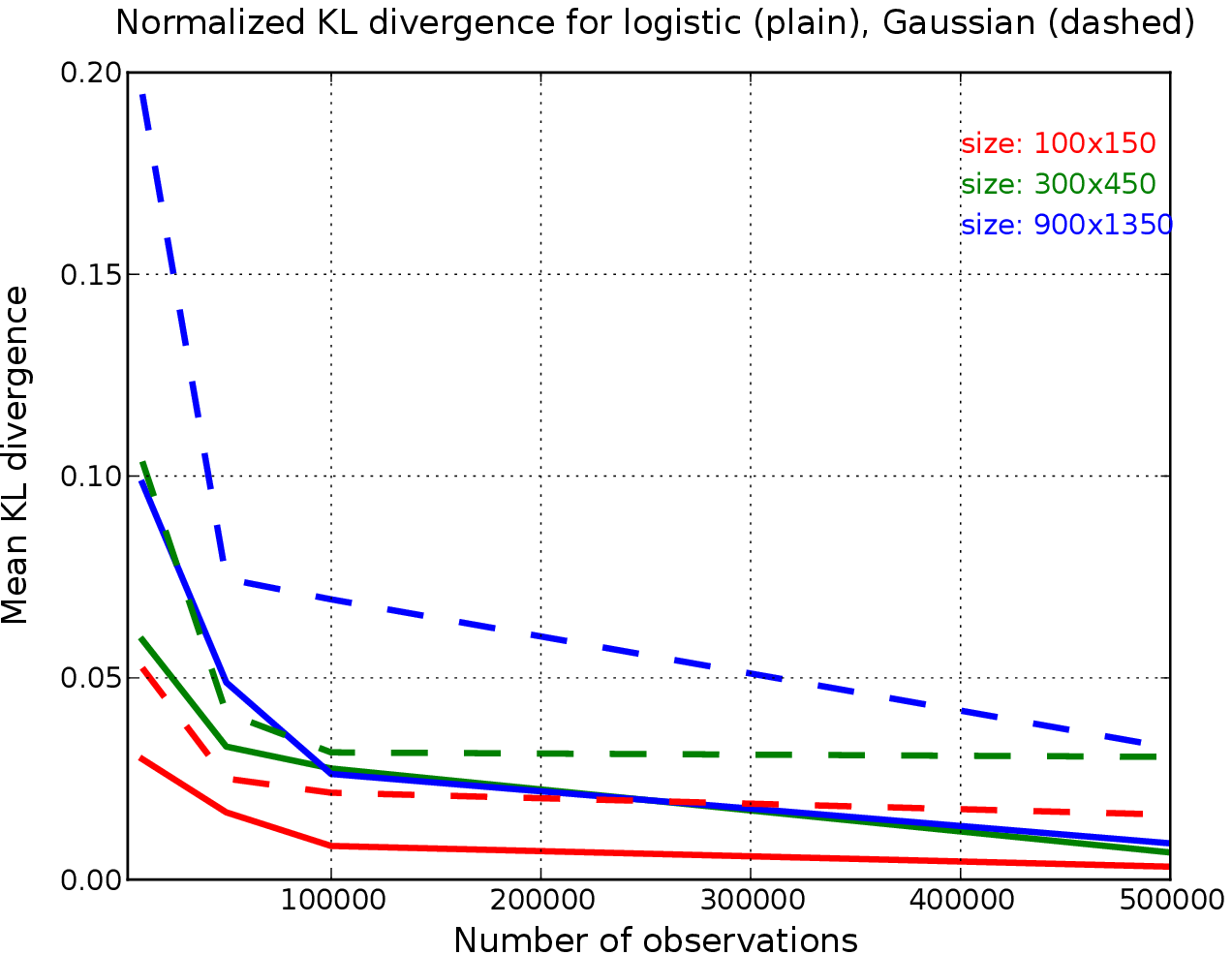}
\end{minipage}\hfill
\begin{minipage}{0.47\textwidth}
 \includegraphics[width=1\textwidth]{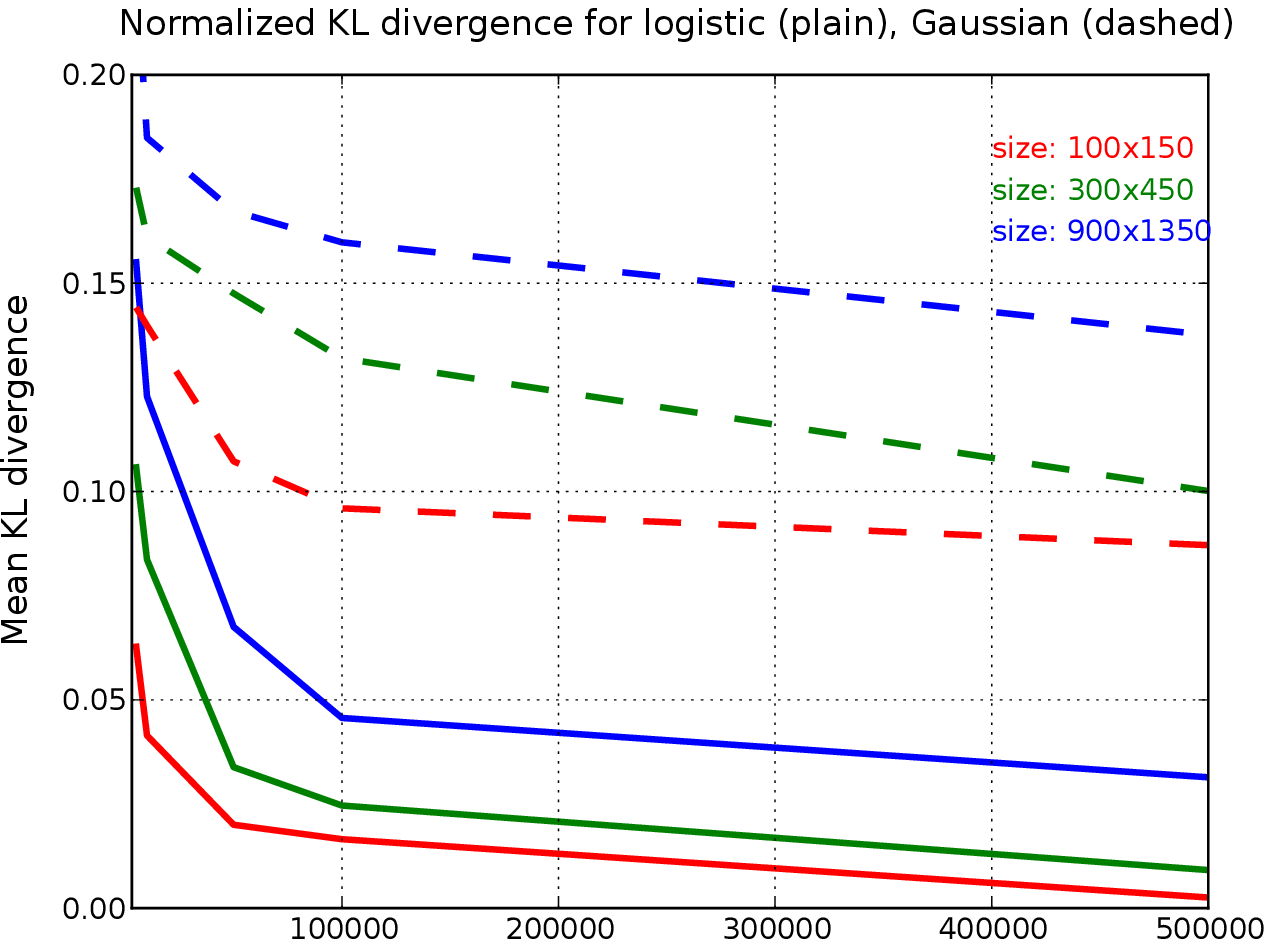}
\end{minipage}

\caption{ Kullback-Leibler divergence between the estimated and the true model
for different matrices sizes and sampling fraction, normalized by number of classes.
Right figure: binomial and Gaussian models ;
left figure: multinomial with five classes and Gaussian model. Results are averaged over five samples.}
\label{fig:KL}
\end{figure}

\begin{table*}[ht!]

\centering
\begin{tabular}{|l|c|c|c|c|}
   \hline
   \textbf{Observations} & $10\cdot 10^3$ &$50\cdot 10^3$ &$100\cdot 10^3$ &$500\cdot 10^3$   \\
   \hline
   \textbf{Gaussian prediction error} & $0.49$ &$0.34$ & $0.29$ & $0.26$  \\
   \hline
   \textbf{Logistic prediction error}& $0.42$ & $0.30$ & $0.27$ & $0.24$  \\
   \hline
\end{tabular}
\caption{Prediction errors for a binomial (2 classes) underlying model, for a $900 \times 1350$ matrix.}
\label{tab:sim_bin_error}
\end{table*}

\begin{table*}[h]

\centering
\begin{tabular}{|l|c|c|c|c|}
   \hline
   \textbf{Observations} & $10\cdot 10^3$ &$50\cdot 10^3$ &$100\cdot 10^3$ &$500\cdot 10^3$   \\
   \hline
   \textbf{Gaussian prediction error} & $0.78$ &$0.76$ & $0.73$ & $0.69$  \\
   \hline
   \textbf{Logistic prediction error}& $0.75$ & $0.54$ & $0.47$ & $0.43$  \\
   \hline
\end{tabular}
\caption{Prediction Error for a multinomial (5 classes) distribution against a $900 \times 1350$ matrix.}
\label{tab:sim_mult_error}
\end{table*}

\textbf{Real dataset}
\quad We have also run the same estimators on the MovieLens $100k$ dataset. In the case of real data
we cannot calculate the Kullback-Leibler divergence since no ground truth is available.
Therefore, to compare the prediction errors, we randomly selected $20\%$ of the entries as a test set,
and the remaining entries were split between a training set ($80\%$) and a validation set ($20\%$).

For this dataset, ratings range from $1$ to $5$. To consider the benefit of a binomial model, we have tested
each rating against the others (\eg ratings $5$ are set to $0$ and all others are set to $1$).
Interestingly we see that the Gaussian
prediction error is significantly better when choosing labels $-1$, $1$ instead of labels $0$, $1$. This is another motivation
for not using the Gaussian version: the sensibility to the alphabet choice seems to be crucial for the Gaussian version,
whereas the binomial/multinomial ones are insensitive to it.
These results are summarized in table \ref{tab:ml_bin}.

\begin{table*}[h]
\centering
\begin{tabular}{|l|c|c|c|c|c|}
   \hline
   \textbf{Rating} & $1$ &$2$ &$3$ &$4$ & $5$   \\
   \hline
   \textbf{Gaussian prediction error (labels $-1$ and $1$)} & $0.06$ &$0.12$ & $0.28$ & $0.35$ &$0.19$  \\
   \hline
   \textbf{Gaussian prediction error (labels $0$ and $1$)} & $0.12$ &$0.20$ & $0.39$ & $0.46$ &  $0.30$\\
   \hline
   \textbf{Logistic prediction error}& $0.06$ &$0.11$ & $0.27$ & $0.34$ &$0.20$  \\
   \hline
\end{tabular}
\caption{Binomial prediction error when performing one versus the others procedure on the MovieLens $100k$
dataset.}
\label{tab:ml_bin}
\end{table*}

\section{Conclusion and future work}

We have proposed a new nuclear norm penalized maximum
log-likelihood estimator and have provided strong theoretical
guarantees on its estimation accuracy in the binary case. 
Compared to previous works on
1-bit matrix completion, our method has some important advantages.
First, it works under quite mild assumptions on the sampling
distribution. Second, it requires only an upper bound on the maximal
absolute value of the unknown matrix. 
Finally, the rates of convergence given by Theorem~\ref{th2} 
are faster than the rates of convergence obtained in \cite{Davenport_Plan_VandenBerg_Wootters12} and
\cite{Cai_Zhou14}.	
In future work, we could consider the extension to more general data fitting terms, 
and to possibly generalize the results to tensor formulations,
or to penalize directly the nuclear norm of the matrix probabilities themselves.

\subsection*{Acknowledgments}
Jean Lafond is grateful for fundings from the Direction G\'en\'erale de l'Armement (DGA) and 
to the labex LMH through the grant no ANR-11-LABX-0056-LMH in the framework of the
"Programme des Investissements d'Avenir". Joseph Salmon acknowledges 
Chair Machine Learning for Big Data for partial financial support. The authors would also like to thank
Alexandre Gramfort for helpful discussions.

\bibliographystyle{plain}
\bibliography{references_all}

\section{Appendix}
\subsection{Proof of Theorem \ref{th1}}
\label{proofth1}

We consider a matrix $\mat{X}$ which satisfies $\Obj(\mat{X}) \leq \Obj(\mat{\tX})$, (\eg $X=\hat{\mat{X}}$).
Recalling $\rth:= (2 m_1 m_2\rank(\mat{\tX}))/K_\gamma$, we get from Lemma \ref{InegBeg}
\begin{equation}
\label{th1:ineg1}
\Lik(\mat{X})-\Lik(\mat{\tX})\leq \lambda \sqrt{\rth } \sqrt{\KL{f(\mat{\tX})}{f(\mat{X})}} \eqs.
\end{equation}
Let us define
\begin{equation}\label{eq:def_D}
\D \left( f(\mat{X'}),f(\mat{X}) \right):= \EE\left[\left(\Lik(\mat{X})-\Lik(\mat{X'})\right) \right]\eqs,
\end{equation}
where the expectation is taken both over the $(E_i)_{1 \leq i \leq n}$ and $(Y_i)_{1 \leq i \leq n}$.
As stated in Lemma \ref{D:K}, Assumption \ref{A1} implies  $\mu \D\left(f(\mat{\tX}),f(\mat{X})\right) \geq   \KL{f(\mat{\tX})}{f(\mat{X})}$.
We now need to control the left hand side of \eqref{th1:ineg1} uniformly over $X$  with high probability. Since we assume $\lambda \geq 2\|\bar{\Sigma}\|_{\sigma,\infty}$
applying \autoref{LemProj}~\eqref{NucKul}  and then \autoref{D:K} yields
\begin{equation}
\label{proof:th1ineq1}
 \|\mat{X}-\mat{\tX}\|_{\sigma,1}  \leq 4 \sqrt{\rth}\sqrt{\KL{f(\mat{\tX})}{f(\mat{X})}} \leq 4 \sqrt{\mu\rth} \sqrt{\D\left(f(\mat{\tX}),f(\mat{X})\right)} \eqs,
\end{equation}
Consequently, if we define $ \mathcal{C}(r)$ as
\begin{equation*}
 \mathcal{C}(r):=\left\{\mat{X} \in  \RR^{m_1 \times m_2} : \:\: \|\mat{X}\|_{\infty}\leq \gamma, \: \|\mat{X}-\mat{\tX}\|_{\sigma,1}^2\leq r \D\left(f(\mat{\tX}),f(\mat{X})\right) \right\}\eqs,
\end{equation*}
we need to control $(\Lik(\mat{X})-\Lik(\mat{\tX}))$ for $\mat{X} \in \mathcal{C}(16 \mu \rth)$.
For technical reasons, we have to ensure that $\D\left(f(\mat{\tX}),f(\mat{X})\right)$ is
greater than a given threshold $\beta >0$ and therefore we define the following set
\begin{equation*}
 \mathcal{C}_\beta(r)=\left\{ \mat{X} \in  \RR^{m_1 \times m_2} : \:\: \mat{X} \in \mathcal{C}(r), \: \D\left(f(\mat{\tX}),f(\mat{X})\right) > \beta \right\}\eqs.
\end{equation*}
We then distinguish the two following cases.\\
\textbf{Case 1}. If $\D\left(f(\mat{\tX}),f(\mat{X})\right) > \beta$, \eqref{proof:th1ineq1} gives $X \in \mathcal{C}_\beta(16\mu\rth)$.
Plugging \autoref{DevUnifCont} in \eqref{th1:ineg1} with $\beta=2M_\gamma \sqrt{\log(d)}/(\eta\sqrt{n\log(\alpha)})$
, $\alpha=e$ and $\eta=1/(4\alpha)$ with probability at least
$1-2d^{-1}/(1-d^{-1})\geq 1-2/d$ it holds
\begin{equation*}
\frac{\D\left(f(\mat{\tX}),f(\mat{X})\right)}{2}-\epsilon(16\mu\rth,\alpha,\eta) \leq \lambda \sqrt{\rth} \sqrt{\KL{f(\mat{\tX})}{f(\mat{X})}}\eqs,
\end{equation*}
where $\epsilon$ is defined in \autoref{DevUnifCont}.
Recalling \autoref{D:K} we get
 \begin{equation*}
  \frac{\KL{f(\mat{\tX})}{f(\mat{X})}}{2\mu} -\lambda \sqrt{\rth } \sqrt{\KL{f(\mat{\tX}}{f(\mat{X})}} -\epsilon(16\mu\rth,\alpha,\eta) \leq 0\eqs.
 \end{equation*}
An analysis of this second order polynomial and $\epsilon(16\mu\rth,\alpha,\eta)/\mu=\epsilon(16\rth,\alpha,\eta)$ leads to

\begin{equation}
 \label{th1bound1}
 \sqrt{\KL{f(\mat{\tX})}{f(\mat{X})}} \leq \mu  \left( \lambda \sqrt{\rth } + \sqrt{\lambda^2\rth + 2 \epsilon(16\rth,\alpha,\eta)}\right) \eqs,
\end{equation}
from which we derive the first bound of the \autoref{th1}.\\
\textbf{Case 2}. If $\D\left(f(\mat{\tX}),f(\mat{X})\right)\leq \beta$ then Lemma \ref{D:K} yields
\begin{equation}
 \label{th1bound2}
 \KL{f(\mat{\tX})}{f(\mat{X})}\leq \mu \beta \eqs.
\end{equation}
Combining \eqref{th1bound1} and \eqref{th1bound2} concludes the proof. \hfill $\qed$

\subsection{Proof of Theorem \ref{th2}}
\label{proofth2}
By \autoref{HelFro}, one only needs to prove the upper bound for the Kullback Leibler divergence.
The main points in proving \autoref{th2} is controlling $\| \bar{\Sigma}\|_{\sigma,\infty}$ and $\EE \|\Sigma_R \|_{\sigma,\infty}$. By definition
 \begin{equation*}
  \bar{\Sigma}=- \sum_{i=1}^{n}\left[ \left(\1_{\{Y_i=1\}} \frac{f'(\langle \mat{\tX}|E_i \rangle)}{f(\langle \mat{\tX}|E_i \rangle)}
  -\1_{\{Y_i=2\}}\frac{f'(\langle \mat{\tX}|E_i \rangle)}{1-f(\langle \mat{\tX}|E_i \rangle)}\right)E_i \right]\eqs.
 \end{equation*}
For $i\in [n]$, the matrices $Z_i:=(\1_{\{Y_i=1\}} \frac{f'(\langle \mat{\tX}|E_i \rangle)}{f(\langle \mat{\tX}|E_i \rangle)}
  -\1_{\{Y_i=2\}}\frac{f'(\langle \mat{\tX}|E_i \rangle)}{1-f(\langle \mat{\tX}|E_i \rangle)})E_i$ are independent, and satisfy $\EE[Z_i]=0$ as a score function.
  Moreover one can check that $\| \mat{Z_{i}} \|_{\sigma,\infty} \leq L_\gamma$. Noticing $E_{k,l} E_{k,l}^\top =E_{k,k}$  we also get
  \begin{multline*}
\sum_{i=1}^{n} \EE[ \mat{Z_{i}}\mat{Z_{i}}^\top ]=\\
    \sum_{k=1}^{m_1}\left( \sum_{l=1}^{m_2}\pi_{k,l} \left( f(\mat{\tX}_{k,l}) \frac{f'^2( \mat{\tX}_{k,l})}{f^2(\mat{\tX}_{k,l})}
   +(1-f(\mat{\tX}_{k,l})\frac{f'^2(\mat{\tX}_{k,l})}{(1-f(\mat{\tX}_{k,l}))^2} \right)\right)E_{k,k} \eqs,
  \end{multline*}
which is diagonal. Since $f$ takes value in $[0,1]$, for any $(k,l) \in [m_1] \times [m_2] $ it holds
\begin{equation*}
 f(\mat{\tX}_{k,l}) \frac{f'^2( \mat{\tX}_{k,l})}{f^2(\mat{\tX}_{k,l})}
   +(1-f(\mat{\tX}_{k,l})\frac{f'^2(\mat{\tX}_{k,l})}{(1-f(\mat{\tX}_{k,l}))^2} \leq L^2\gamma \eqs,
\end{equation*}
so that we obtain
\begin{equation*}
\left\| \EE[\frac{1}{n} \sum_{i=1}^{n}  \mat{Z_{i}}\mat{Z_{i}}^\top ]\right\|_{\sigma,\infty}\leq L_\gamma^2 \max_{l}(C_l) \leq L_\gamma^2 \frac{\Lc}{m} \eqs,
\end{equation*}
were we used Assumption \ref{A2} for the last inequality.
We show similarly that $\|\EE[\sum_{i=1}^{n} \mat{Z_{i}}^\top\mat{Z_{i}} ]\|_{\sigma,\infty}/n\leq L_\gamma^2 \Lc/m$.
Therefore, Proposition \ref{prop:bernstein} applied with $t=\log(d)$, $U=L_\gamma$ and $\sigma_Z^2 =L^2_\gamma\Lc/m$
yields with at least probability $1-1/d$,
 \begin{equation}
 \label{proof:th2ineq1}
  \|\bar{\Sigma}\|_{\sigma,\infty} \leq c^* L_\gamma\max \left\{ \sqrt{ \frac{2\Lc\log(d)}{mn}} , \frac{2}{3} \frac{\log(d)}{n} \right \} \eqs.
  \end{equation}
With the same analysis for $\Sigma_R:=\frac{1}{n}\sum_{i=1}^{n} \varepsilon_i E_i$ and by applying \autoref{lem:MatExp}
with $U=1$ and $\sigma_Z^2 =\frac{L}{m}$, for $n\geq n^* :=m\log(d)/(9 \Lc)$ it holds:
\begin{equation}
 \label{proof:th2ineq2}
  \EE \left[\|\Sigma_R  \|_{\sigma,\infty} \right] \leq c^{*}\sqrt{\frac{2e\Lc \log(d)}{mn}} \eqs.
\end{equation}
Assuming $n \geq 2 m \log(d)/(9L)$, implies $n \geq n^*$ and \eqref{proof:th2ineq2} is therefore satisfied.
Since it also implies $ \sqrt{2L\log(d)/(mn)} \geq 2\log(d)/(3n)$, the second term of \eqref{proof:th2ineq1} is negligible.
Consequently taking $\lambda:=2c^* L_\gamma\sqrt{2L\log(d)/(mn)}$ ensures that $\lambda \geq  2\|\bar{\Sigma}\|_{\sigma,\infty} $
with at least probability $1-1/d$.\\
Therefore
 by taking $\lambda$, $\beta$ and $n$  as in \autoref{th2} statement , with at least probability $1-3/d$,
 \autoref{th1} result holds when replacing  $\EE \|\Sigma_R  \|_{\sigma,\infty} $ by its upper bound \eqref{proof:th2ineq2}, which is
 exactly \autoref{th2} statement.
 \hfill $\qed$
\subsection{Linear Algebra and Distance Comparison}
We denote by $\mathcal{S}_1(\mat{X}) \subset \RR^{m_1}$ (\resp $\mathcal{S}_2(\mat{X}) \subset \RR^{m_2}$)
the linear spans generated by left (\resp right) singular vectors of $\mat{X}$.
$P_{\mathcal{S}^\bot_1(\mat{X})}$  (\resp $P_{\mathcal{S}^\bot_2(\mat{X})}$)
denote the orthogonal projections on $\mathcal{S}^\bot_1(\mat{X})$ (\resp $\mathcal{S}^\bot_2(\mat{X})$).
We then define the following orthogonal projections on $\RR^{m_1 \times m_2}$
\begin{equation*}
\Proj_{\mat{X}}^\bot:\mat{X'} \to P_{\mathcal{S}^\bot_1(\mat{X})}\mat{X'} P_{\mathcal{S}^\bot_2(\mat{X})}
\text{ and } \Proj_{\mat{X}}\mat{X'} \to \mat{X'}-\Proj_{\mat{X}}^\bot(\mat{X'})\eqs.
\end{equation*}

\begin{lemma}
\label{hel:kul}
 For any matrices $\mat{X},\mat{X'} \in \RR^{m_1 \times m_2}$ it holds:
\begin{equation*}
  \dhe^2(f(\mat{X}),f(\mat{X'})) \leq \KL{f(\mat{X})}{f(\mat{X'})}
\end{equation*}
\end{lemma}

\begin{proof}
See \cite[Lemma 4.2]{Tsybakov09}
\end{proof}

\begin{lemma}
\label{algebricres}
 For any matrix $\mat{X}$ and $\mat{X'}$ we have
 \begin{enumerate}[(i)]
  \item $ \|\mat{X} + \Proj_{\mat{X}}^\bot(\mat{X'}) \|_{\sigma,1}=\|\mat{X}\|_{\sigma,1} +  \|\Proj_{\mat{X}}^\bot(\mat{X'}) \|_{\sigma,1}\eqs,$ \label{PenRel}
  \item $\| \Proj_{\mat{X}}(\mat{X'}) \|_{\sigma,1} \leq \sqrt{2 \rank(\mat{X})} \|\mat{X'}\|_{\sigma,2}\eqs,$ \label{ProjRel}
  \item  $ \|\mat{X}\|_{\sigma,1}-\|\mat{X'}\|_{\sigma,1} \leq  \|\Proj_{\mat{X}}(\mat{X'}-\mat{X})\|_{\sigma,1}\eqs.$ \label{diffschatten}
  \end{enumerate}

\end{lemma}
\begin{proof}
The proof is straightforward and is left to the reader.
\end{proof}

\begin{lemma}
\label{HelFro}
 For any $\gamma>0$, there exist a constant $K_\gamma>0$ such that for any $\mat{X},\mat{X'}  \in  \RR^{m_1 \times m_2}$
satisfying $\|\mat{X}\|_{\infty}\leq \gamma$ and $\|\mat{X'}\|_{\infty}\leq \gamma$, such that the following holds:
 \begin{equation*}
 \|\mat{X}-\mat{X'}\|^2_{\sigma,2} \leq \frac{m_1 m_2}{K_\gamma} \dhe^2(f(\mat{X}),f(\mat{X'}))\leq \frac{m_1 m_2}{K_\gamma} \KL{f(\mat{X})}{f(\mat{X'})} \eqs.
 \end{equation*}
\end{lemma}

\begin{proof}
A proof is given in \cite[Lemma 2]{Davenport_Plan_VandenBerg_Wootters12} and we provide it here for self-completeness.
The second inequality is a consequence of the first one and \autoref{hel:kul}. Let $x,y \in [-\gamma,\gamma]$. We have
 \begin{multline*}
 \left(\sqrt{f(x)}-\sqrt{f(y)}\right)^2 +  \left(\sqrt{1-f(x)}-\sqrt{1-f(y)}\right)^2\geq  \\
\frac{1}{2} \left( \sqrt{f(x)}-\sqrt{1-f(x)}-\sqrt{f(y)} +\sqrt{1-f(y)} \right)^2\eqs.
 \end{multline*}

The mean value theorem applied to $x \rightarrow \sqrt{f(x)}-\sqrt{1-f(x)}$ implies the existence of $c_{x,y} \in [-\gamma,\gamma] $ such that
\begin{multline*}
 \left(\sqrt{f(x)}-\sqrt{f(y)}\right)^2 +   \left(\sqrt{1-f(x)}-\sqrt{1-f(y)}\right)^2=
 \\ \frac{\left(f'(c_{x,y})\right)^2}{8f(c_{x,y})(1-f(c_{x,y}))}\left(\sqrt{f(c_{x,y})}+\sqrt{1-f(c_{x,y})}\right)^2\left(x-y\right)^2\eqs,
\end{multline*}
The proof is concluded by noting that $u \rightarrow \sqrt{u}+\sqrt{1-u}$ is
lower bounded by $1$ on $[0,1]$.
\end{proof}

\begin{lemma}
 \label{InegBeg}
 Let $\mat{X},\mat{X'} \in \RR^{m_1 \times m_2}$ satisfying  $\Obj(\mat{X}) \leq \Obj(\mat{X'})$, then
 \begin{equation*}
\Lik(\mat{X})-\Lik(\mat{X'})\leq \lambda \sqrt{\frac{2m_1m_2 \rank(\mat{X'})}{K_\gamma} } \sqrt{\KL{f(\mat{X'})}{f(\mat{X})}} \eqs.
 \end{equation*}
\end{lemma}

\begin{proof}
Since $\Obj(\mat{X}) \leq \Obj(\mat{X'})$, we obtain
 \begin{align*}
\Lik(\mat{X})-\Lik(\mat{X'}) \leq &\lambda (\|\mat{X'}\|_{\sigma,1}-\|\mat{X}\|_{\sigma,1})\leq \lambda \|\Proj_{\mat{X'}}(\mat{X}-\mat{X'})\|_{\sigma,1}
\eqs, \\
					    \leq &\lambda \sqrt{2 \rank(\mat{X'}) } \|\mat{X}-\mat{X'}\|_{\sigma,2}\eqs,\\
\end{align*}
where we have used \autoref{algebricres} \eqref{diffschatten} and \eqref{ProjRel} for the last two lines and
 \autoref{HelFro} and \autoref{hel:kul} to get the result.
\end{proof}

\begin{lemma}
\label{LemProj}
Let $\mat{X},\mat{X'} \in \RR^{m_1 \times m_2}$ satisfying $\|\mat{X}\|_{\infty}\leq \gamma$ and $\|\mat{X'}\|_{\infty}\leq \gamma$ and $\lambda>2\| \Sigma_Y(\mat{X'}) \|_{\sigma,\infty}$.
Assume that $\Obj(\mat{X}) \leq \Obj(\mat{X'})$. Then
  \begin{enumerate}[(i)]
 \item $\|\Proj_{\mat{X'}}^\bot(\mat{X}-\mat{X'})\|_{\sigma,1} \leq 3 \|\Proj_{\mat{X'}}(\mat{X}-\mat{X'})\|_{\sigma,1}\eqs,$ \label{projrel}
\item $\|\mat{X}-\mat{X'}\|_{\sigma,1}  \leq 4 \sqrt{2 \rank(\mat{X'})} \|(\mat{X}-\mat{X'})\|_{\sigma,2}\eqs,$\label{NucFob}
\item  $\|\mat{X}-\mat{X'}\|_{\sigma,1} \leq 4 \sqrt{ 2 m_1 m_2  \rank(\mat{X'})/ K_\gamma}   \dhe\left(f(\mat{X'}),f(\mat{X})\right)\eqs,$\label{NucHel}
 \item $\|\mat{X}-\mat{X'}\|_{\sigma,1} \leq 4 \sqrt{ 2 m_1 m_2  \rank(\mat{X'})/ K_\gamma }  \sqrt{\KL{f(\mat{X'})}{f(\mat{X})}}\eqs.$\label{NucKul}
\end{enumerate}
\end{lemma}

\begin{proof}
We first prove \eqref{projrel}. Since $\Obj(\mat{X}) \leq \Obj(\mat{X'})$, we have
 \begin{equation*}
- ( \Lik(\mat{X}) -\Lik(\mat{X'}) ) \geq \lambda (\|\mat{X}\|_{\sigma,1}-\|\mat{X'}\|_{\sigma,1}).
 \end{equation*}
Writing $\mat{X}=\mat{X'}+\Proj_{\mat{X'}}^\bot(\mat{X}-\mat{X'})+\Proj_{\mat{X'}}(\mat{X}-\mat{X'})$ and using Lemma \ref{algebricres} \eqref{PenRel} and the triangular inequality we get
\begin{equation*}
 \|\mat{X}\|_{\sigma,1} \geq \|\mat{X'}\|_{\sigma,1} + \|\Proj_{\mat{X'}}^\bot(\mat{X}-\mat{X'})\|_{\sigma,1}-\|\Proj_{\mat{X'}}(\mat{X}-\mat{X'})\|_{\sigma,1} \eqs,
\end{equation*}
which implies
\begin{equation}
 \label{LemProj:lb}
- (\Lik(\mat{X})-\Lik(\mat{X'})) \geq  \lambda \left( \|\Proj_{\mat{X'}}^\bot(\mat{X}-\mat{X'})\|_{\sigma,1}-\|\Proj_{\mat{X'}}(\mat{X}-\mat{X'})\|_{\sigma,1} \right)\eqs.
\end{equation}
Furthermore by concavity of $\Lik$ we have
\begin{align*}
-(\Lik(\mat{X}) -\Lik(\mat{X'})) \leq \langle \Sigma_Y(\mat{X'})|\mat{X'}-\mat{X} \rangle \eqs.
\end{align*}
The duality between $\|\cdot\|_{\sigma,1}$ and $\|\cdot\|_{\sigma,\infty}$ (see for instance
\cite[Corollary IV.2.6]{Bhatia97}) leads to
\begin{align}
- (\Lik(\mat{X})-\Lik(\mat{X'})) &\leq \|\Sigma_Y(\mat{X'})\|_{\sigma,\infty}\|\mat{X'}-\mat{X} \|_{\sigma,1}\eqs,\nonumber \\
& \leq \frac{\lambda}{2}\|\mat{X'}-\mat{X} \|_{\sigma,1}\eqs, \nonumber \\
& \leq \frac{\lambda}{2} (\|\Proj_{\mat{X'}}^\bot(\mat{X}-\mat{X'})\|_{\sigma,1} + \|\Proj_{\mat{X'}}(\mat{X}-\mat{X'})\|_{\sigma,1})\eqs, \label{LemProj:ub}
\end{align}
where we used $\lambda>2\| \Sigma_Y(\mat{X'}) \|_{\sigma,\infty}$ in the second line. Then combining \eqref{LemProj:lb} with \eqref{LemProj:ub} gives \eqref{projrel}.\\
Since $\mat{X}-\mat{X'}=\Proj_{\mat{X'}}^\bot(\mat{X}-\mat{X'})+\Proj_{\mat{X'}}(\mat{X}-\mat{X'})$, using the triangular inequality and \eqref{projrel} yields
\begin{equation}
\label{LemProj:int}
 \|\mat{X}-\mat{X'}\|_{\sigma,1} \leq 4 \|\Proj_{\mat{X'}}(\mat{X}-\mat{X'})\|_{\sigma,1}.
\end{equation}
Combining \eqref{LemProj:int} and  \eqref{projrel} immediately leads to \eqref{NucFob} and \eqref{NucHel} is a consequence of  \eqref{NucFob} and \autoref{HelFro}. The
statement \eqref{NucKul} follows from \eqref{NucHel} and \autoref{hel:kul}.
\end{proof}

\subsection{Likelihood Deviation}

\begin{lemma}
\label{D:K}
 Under Assumption \ref{A1} we have
 \begin{equation*}
 \D\left(f(\mat{\tX}),f(\mat{X})\right) \geq  \frac{1}{\mu} \KL{f(\mat{\tX})}{f(\mat{X})} \eqs.
\end{equation*}
where $\D(\cdot,\cdot)$ is defined in Eq.~\eqref{eq:def_D}.
\end{lemma}
\begin{proof}
% Recalling $\pi_{k,l}= \Pi(E_1=E_{k,l})$, we have
\begin{align*}
 &\D\left(f(\mat{\tX}),f(\mat{X})\right) \\
  &=\sum_{i=1}^{n}\sum_{\substack{1 \leq k \leq m_1 \\1 \leq l \leq m_2 }} \pi_{k,l}\left(f(\mat{\tX}_{k,l})\log\left(\frac{f(\mat{\tX}_{k,l})}{f(\mat{X}_{k,l})}\right)
+(1-f(\mat{\tX}_{k,l}))\log\left(\frac{1-f(\mat{\tX}_{k,l})}{f(1-\mat{X}_{k,l})}\right) \right)\eqs,\\
&\geq \frac{1}{\mu m_1 m_2} \sum_{i=1}^{n}\sum_{\substack{1 \leq k \leq m_1 \\1 \leq l \leq m_2 }} \left(f(\mat{\tX}_{k,l})\log\left(\frac{f(\mat{\tX}_{k,l})}{f(\mat{X}_{k,l})}\right)
+(1-f(\mat{\tX}_{k,l}))\log\left(\frac{1-f(\mat{\tX}_{k,l})}{f(1-\mat{X}_{k,l})}\right) \right)\eqs,
\end{align*}
where $\pi_{k,l}$ is given by Eq. \eqref{eq:def_D}.
\end{proof}
% We now recall the definitions of the following sets
% \begin{align*}
%  \mathcal{C}(r):=&\left\{\mat{X} \in  \RR^{m_1 \times m_2} : \:\: \|\mat{X}\|_{\infty}\leq \gamma, \: \|\mat{X}-\mat{\tX}\|_{\sigma,1}^2\leq r \D\left(f(\mat{\tX}),f(\mat{X})\right) \right\}\eqs, \\
%   \mathcal{C}_\beta(r):=&\left\{ \mat{X} \in  \RR^{m_1 \times m_2} : \:\: \mat{X} \in \mathcal{C}(r), \: \D\left(f(\mat{\tX}),f(\mat{X})\right) > \beta \right\}.
% \end{align*}

 \begin{lemma}
 \label{DevUnifCont}
 Assume that $\lambda \geq \bar{\Sigma}$.  Let $\alpha>1$, $\beta>0$ and $0<\eta<1/2\alpha$.
Then with probability at least
  $1-2(\exp(-n\eta^2\log(\alpha)\beta^2/(4M_\gamma^2))/(1-\exp(-n\eta^2\log(\alpha)\beta^2/(4M_\gamma^2)))$
  we have for all $\mat{X} \in \mathcal{C}_\beta (r)$:
 \begin{equation*}
   |(\Lik(\mat{X})-\Lik(\mat{\tX})) - \D\left(f(\mat{\tX}),f(\mat{X})\right)|\leq \frac{\D\left(f(\mat{\tX}),f(\mat{X})\right)}{2}+\epsilon(r,\alpha,\eta)\eqs ,
 \end{equation*}
where
\begin{equation}\label{eq:def_eps}
 \epsilon(r,\alpha,\eta) :=\frac{4 L_\gamma^2r}{1/(2\alpha)-\eta}(\EE \|\Sigma_R \|_{\sigma,\infty})^2\eqs.
\end{equation}

\end{lemma}

\begin{proof}
 To prove this result we use a peeling argument combined to Lemma \ref{SliceCont}.
 Let us define $\Dn\left(f(\mat{X}),f(\mat{\tX})\right):= -(\Lik(\mat{X})-\Lik(\mat{\tX}))$, and the event
\begin{multline*}
 \mathcal{B} := \bigg\{ \exists \mat{X} \in \mathcal{C}_\beta (r)|\\
  |\Dn\left(f(\mat{X}),f(\mat{\tX})\right) - \D\left(f(\mat{\tX}),f(\mat{X})\right) |
 > \frac{\D\left(f(\mat{\tX}),f(\mat{X})\right)}{2}+\epsilon(r,\alpha,\eta) \bigg\}\eqs,
\end{multline*}
and
\begin{equation*}
 \mathcal{S}_l:= \left\{ \mat{X} \in \mathcal{C}_\beta (r)| \alpha^{l-1}\beta < \D\left(f(\mat{\tX}),f(\mat{X})\right) < \alpha^{l}\beta \right\}\eqs.
\end{equation*}
Let us also define the set
\begin{equation*}
  \mathcal{C}_\beta (r,t)=\left\{ \mat{X} \in  \RR^{m_1 \times m_2}|\:\: \mat{X}  \in \mathcal{C}_\beta (r),\:  \D\left(f(\mat{\tX}),f(\mat{X})\right) \leq t \right\} \eqs,
\end{equation*}
and
\begin{equation}\label{eq:def_Zt}
 Z_t:=\sup_{\mat{X} \in \mathcal{C}_\beta (r,t)} |\Dn\left(f(\mat{X}),f(\mat{\tX})\right) - \D\left(f(\mat{\tX}),f(\mat{X})\right) | \eqs,
\end{equation}
Then for any $\mat{X} \in  \mathcal{B} \cap \mathcal{S}_l$ we have
\begin{equation*}
 |\Dn\left(f(\mat{X}),f(\mat{\tX})\right) - \D\left(f(\mat{\tX}),f(\mat{X})\right) | > \frac{1}{2}\alpha^{l-1}\beta + \epsilon(r,\alpha,\eta)\eqs,
\end{equation*}
Moreover by definition of $\mathcal{S}_l$, $\mat{X} \in \mathcal{C}_\beta (r,\alpha^{l}\beta)$.
Therefore

\begin{equation*}
 \mathcal{B} \cap \mathcal{S}_l \subset \mathcal{B}_l := \{ Z_{\alpha^{l}\beta} > \frac{1}{2\alpha}\alpha^{l}\beta+\epsilon(r,\alpha,\eta) \}\eqs,
\end{equation*}
If we now apply the union bound and \autoref{SliceCont} we get
 \begin{align*}
 \PP(\mathcal{B}) &\leq \sum_{l=1}^{+\infty} \PP(\mathcal{B}_l)\eqs,\\
&\leq \sum_{l=1}^{+\infty} \exp(-\frac{n\eta^2(\alpha^{l}\beta)^2}{8M_\gamma^2})\eqs,\\
  &\leq \frac{\exp(-\frac{n\eta^2\log(\alpha)\beta^2}{4M_\gamma^2})}{1-\exp(-\frac{n\eta^2\log(\alpha)\beta^2}{4M_\gamma^2})}\eqs,
 \end{align*}
 where we used $x\leq e^x$ in the second inequality.
\end{proof}

\begin{lemma}
 \label{SliceCont}
Assume that $\lambda \geq \bar{\Sigma}$. Let $\alpha>1$ and $0<\eta<\frac{1}{2\alpha}$. Then we have
\begin{equation*}
 \PP\left(Z_t > \frac{t}{2\alpha} + \epsilon(r,\alpha,\beta) \right) \leq \exp(-\frac{n\eta^2t^2}{8M_\gamma^2})\eqs,
\end{equation*}
where $\epsilon(r,\alpha,\eta)$ is defined in Eq. \eqref{eq:def_eps}.
% \begin{equation*}
%  \epsilon(r,\alpha,\eta) :=\frac{4 L_\gamma^2r}{1/(2\alpha)-\eta}(\EE \|\Sigma_R \|_{\sigma,\infty})^2\eqs.
% \end{equation*}
 \end{lemma}

 \begin{proof}

Using Massart's inequality (\cite[Theorem 9]{Massart00})
we get for a given $0<~\eta<\frac{1}{2\alpha}$:
\begin{equation}
\label{proof:MassartConc}
 \PP(Z_t> \EE[Z_t]+\eta t)\leq \exp(-\frac{\eta^2nt^2}{8M_\gamma^2})\eqs.
\end{equation}
Besides by symmetrization we have
\begin{multline*}
 \EE[Z_t]\leq \\
 2 \EE\left[\sup_{\mat{X} \in \mathcal{C}_\beta (r,t)} \left|\frac{1}{n}\sum_{i=1}^{n}
 \varepsilon_i\left( \1_{\{Y_i=1\}}\log\left(f(\frac{\langle \mat{X}|E_i \rangle}{\langle \mat{\tX}|E_i \rangle})\right) +
 \1_{\{Y_i=2\}}\log\left(\frac{1-f(\langle \mat{X}|E_i \rangle)}{1-f(\langle \mat{\tX}|E_i \rangle)}\right)\right)\right|\right] \eqs,
\end{multline*}
where $\varepsilon:=(\varepsilon_i)_{1 \leq i\leq n}$ is a Rademacher sequence which is independent from both $Y=(Y_i)_{1 \leq i\leq n}$ and $E=(E_i)_{1 \leq i\leq n}$.
Let us define
\begin{equation*}
 \phi_{E_i}(x):=\frac{1}{L_\gamma}\log(\frac{f(x+\langle \mat{\tX}|E_i \rangle)}{f(\langle \mat{\tX}|E_i \rangle)}) \text{ and }
 \tilde{\phi}_{E_i}(x):=\frac{1}{L_\gamma}\log(\frac{1-f(x+\langle \mat{\tX}|E_i \rangle)}{1-f(\langle \mat{\tX}|E_i \rangle)})\eqs.
\end{equation*}
Then, if we denote by $\EE_{E,Y}$ the conditional expectation with respect to $E$ and $Y$, we have
\begin{multline*}
 \EE[Z_t]\leq \\
 2 L_\gamma \EE\EE_{E,Y}\left[\sup_{\mat{X} \in \mathcal{C}_\beta (r,t)} \left|\frac{1}{n}\sum_{i=1}^{n} \varepsilon_i\left( \1_{\{Y_i=1\}}\phi_{E_i}(\langle \mat{X}- \mat{\tX}|E_i \rangle) +
 \1_{\{Y_i=2\}}\tilde{\phi}_{E_i}(\langle \mat{X}- \mat{\tX}|E_i \rangle) \right)\right|\right] \eqs.
\end{multline*}
Let $\psi:\,\mathscr{E}^n \times  \{-1,1\}^n \mapsto \RR$,
\begin{equation*}
 (e,y) \to \EE \left[\sup_{\mat{X} \in \mathcal{C}_\beta (r,t)} \left|\frac{1}{n}\sum_{i=1}^{n} \varepsilon_i\left( \1_{y_i=1}\Phi_{e_i}(\langle \mat{X}- \mat{\tX}|e_i \rangle) +
 \1_{y_i=-1}\tilde{\Phi}_{e_i}(\langle \mat{X}- \mat{\tX}|e_i \rangle) \right)\right|\right],
\end{equation*}
where the expectation is taken over $\varepsilon_1,\dots,\varepsilon_n$.
By independence of the $\varepsilon_i$'s
\begin{equation*}
 \EE[Z_t]\leq 2 L_\gamma \EE [\psi(E,Y)] \eqs.
\end{equation*}
Besides, since the functions $\phi_{e_i}$ and $\tilde{\phi}_{e_i}$ are contractions that vanish at zero,
 by the contraction principle (\cite[Theorem 4.12]{Ledoux_Talagrand91}) we get
for any$(e,y) \in \mathscr{E}^n \times~\{-1,1\}^n$
\begin{equation*}
 \psi(e,y)\leq 2 L_\gamma \EE\left[\sup_{\mat{X} \in \mathcal{C}_\beta (r,t)} \left|\frac{1}{n}\sum_{i=1}^{n} \varepsilon_i \langle \mat{X}- \mat{\tX}|e_i \rangle \right|\right]\eqs,
\end{equation*}
and therefore
\begin{equation*}
 \EE[Z_t]\leq 4 L_\gamma \EE\left[\sup_{\mat{X} \in \mathcal{C}_\beta (r,t)} \left|\frac{1}{n}\sum_{i=1}^{n} \varepsilon_i \langle \mat{X}- \mat{\tX}|E_i \rangle \right|\right]\eqs,
\end{equation*}
Recalling that $\Sigma_R:=\frac{1}{n}\sum_{i=1}^{n} \varepsilon_i E_i$ leads to
\begin{equation*}
 \EE[Z_t] \leq 4 L_\gamma \EE \left[\sup_{\mat{X} \in \mathcal{C}_\beta (r,t)} \left| \langle \mat{X}- \mat{\tX}|\Sigma_R \rangle \right|\right]
 \leq 4 L_\gamma \EE [\|\Sigma_R \|_{\sigma,\infty}] \sqrt{rt}\eqs,
\end{equation*}
where we used the duality between $\|.\|_{\sigma,\infty}$ and $\|.\|_{\sigma,1}$ and also the fact that ${\mat{X} \in \mathcal{C}_\beta (r,t)}$
for the last inequality. Plugging this inequality into \eqref{proof:MassartConc} gives
\begin{equation*}
 \PP(Z_t> 4 L_\gamma \EE [\|\Sigma_R \|_{\sigma,\infty}] \sqrt{rt}+\eta t)\leq \exp(-\frac{\eta^2nt^2}{8M_\gamma^2})\eqs.
\end{equation*}
Since for any $a,b \in \RR  $ and $c>0$, $ab \leq (a^2/c +cb^2)/2$ we have
\begin{equation*}
4 L_\gamma \EE [\|\Sigma_R \|_{\sigma,\infty}] \sqrt{rt}\leq \frac{1}{1/(2\alpha)-\eta}4 L_\gamma^2r\EE [\|\Sigma_R \|_{\sigma,\infty}]^2+(1/(2\alpha)-\eta)t\eqs.
\end{equation*}
we finally get
\begin{equation*}
  \PP(Z_t>  \frac{t}{2\alpha} + \epsilon(r,\alpha,\eta) ) \leq \exp(-\frac{n\eta^2t^2}{8M_\gamma^2})\eqs,
\end{equation*}
where
\begin{equation*}
 \epsilon(r,\alpha,\eta) :=\frac{1}{1/(2\alpha)-\eta}4 L_\gamma^2r\EE [\|\Sigma_R \|_{\sigma,\infty}]^2\eqs.
\end{equation*}
 \end{proof}
\subsection{Deviation of Matrices}

\begin{proposition}
\label{prop:bernstein}
 Consider a finite sequence of independent random matrices $(\Bern{Z_{i}})_{1 \leq i \leq n}\in \RR^{m_1 \times m_2}$ satisfying $\EE[\Bern{Z_{i}}]=0$ and for some $U>0$,
 $\| \Bern{Z_{i}} \|_{\sigma,\infty} \leq U$ for all $i= 1, \dots, n$. Then for any $t>0$
 \begin{equation*}
 \PP\left( \left\|\frac{1}{n} \sum_{i=1}^{n} \Bern{Z_{i}}  \right\|_{\sigma,\infty} > t \right) \leq d\exp (- \frac{nt^2/2}{\sigma^2_Z+Ut/3}) \eqs,
 \end{equation*}
where $d=m_1+m_2$ and
\begin{equation*}
 \sigma^2_Z := \max \left\{\left\| \frac{1}{n} \sum_{i=1}^{n} \EE[  \Bern{Z_{i}}\Bern{Z_{i}^\top} ] \right\|_{\sigma,\infty},
 \left\| \frac{1}{n}\sum_{i=1}^{n} \EE[ \Bern{Z_{i}}^\top\Bern{Z_{i}} ]\right\|_{\sigma,\infty}\right\} \eqs.
\end{equation*}
 In particular it implies that with at least probability $1-e^{-t}$
 \begin{equation*}
  \|\frac{1}{n} \sum_{i=1}^{n} \Bern{Z_{i}}  \|_{\sigma,\infty} \leq c^* \max \left\{ \sigma_Z \sqrt{ \frac{t + \log(d)}{n}} ,  \frac{U(t + \log(d))}{3n} \right \} \eqs,
 \end{equation*}
 with $c^*:=1+\sqrt{3}$.
\end{proposition}
\begin{proof}
The first claim of the proposition is Bernstein's inequality for random matrices (see for example
\cite[Theorem 1.6]{Tropp12}).
Solving the equation (in $t$) $- \frac{nt^2/2}{\sigma^2_Z+Ut/3} + \log(d)=-v$ gives with at least probability $1-e^{-v}$
 \begin{equation*}
  \|\frac{1}{n} \sum_{i=1}^{n} \Bern{Z_{i}}  \|_{\sigma,\infty} \leq  \left[\frac{U}{3}(v + \log(d))+\sqrt{\frac{U^2}{9}(v + \log(d))^2+2n\sigma_Z^2(v + \log(d))}\right]/n\eqs,
 \end{equation*}
 we conclude the proof by distinguishing the two cases $n\sigma_Z^2 \leq \frac{U^2}{9}(v + \log(d))$ or $n\sigma_Z^2 > \frac{U^2}{9}(v + \log(d))$.
\end{proof}

\begin{lemma}
 \label{lem:MatExp}
Let $h \geq 1$. With the same assumptions as \autoref{prop:bernstein},
assume $n \geq n^*:=(U^2\log(d))/(9\sigma^2_Z)$ then the following holds:
\begin{equation*}
  \EE \left[\|\frac{1}{n} \sum_{i=1}^{n} \Bern{Z_{i}}  \|_{\sigma,\infty}^h \right] \leq \left(\frac{2ehc^{*2}\sigma_Z^2\log(d)}{n}\right)^{h/2} \eqs.
\end{equation*}
\end{lemma}

\begin{proof}
For self-completeness we give the proof which is the same as in \cite[Lemma 6]{Klopp14}.
Let us define $t^*:= \frac{9n\sigma_Z^2}{U^2}-\log(d)$ the value of $t$ for which the two bounds of \autoref{prop:bernstein} are equal.
Let $\nu_1:=n/(\sigma_Z^2c^{*2})$ and $\nu_2:=3n/(Uc^*)$ then, from \autoref{prop:bernstein} we have
\begin{align*}
\PP \left( \|\frac{1}{n} \sum_{i=1}^{n} \Bern{Z_{i}}  \|_{\sigma,\infty} > t \right) &\leq d \exp(-\nu_1t^2) \text{ for } t \leq t^* \eqs, \\
\PP \left( \|\frac{1}{n} \sum_{i=1}^{n} \Bern{Z_{i}}  \|_{\sigma,\infty} > t \right) &\leq d\exp(-\nu_2t) \text{ for } t \geq t^* \eqs,
\end{align*}
Let $h\geq 1$, then
\begin{align*}
 &\EE \left[\|\frac{1}{n} \sum_{i=1}^{n} \Bern{Z_{i}}  \|_{\sigma,\infty}^h \right] \eqs,\\
 &\leq \EE \left[\|\frac{1}{n} \sum_{i=1}^{n} \Bern{Z_{i}}  \|_{\sigma,\infty}^{2h\log(d)} \right]^{1/(2\log(d))} \eqs, \\
 &\leq \left( \int_0^{+\infty} \PP \left( \|\frac{1}{n} \sum_{i=1}^{n} \Bern{Z_{i}}  \|_{\sigma,\infty} > t^{1/(2h\log(d))} \right) \right)^{1/(2\log(d))}\eqs, \\
 & \leq d^{1/(2h\log(d))} \left( \int_0^{+\infty} \exp(-\nu_1t^{2/(2h\log(d))}) +\int_0^{+\infty} \exp(-\nu_2t^{1/(2h\log(d))})  \right)^{1/(2\log(d))} \eqs,\\
 &\leq  \sqrt{e} \left( h\log(d) \nu_1^{-h\log(d)}\Gamma(h\log(d)) +2h\log(d) \nu_2^{-2h\log(d)}\Gamma(2h\log(d))  \right)^{1/(2\log(d))} \eqs,
\end{align*}
where we used Jensen's inequality for the first line. Since Gamma-function satisfies for $x\geq 2$, $\Gamma(x)\leq (\frac{x}{2})^{x-1}$ (see \cite[Proposition 12]{Klopp11b}) we have
\begin{align*}
 &\EE \left[\|\frac{1}{n} \sum_{i=1}^{n} \Bern{Z_{i}}  \|_{\sigma,\infty}^h \right]\eqs,  \\
 &\leq \sqrt{e} \left( (h\log(d))^{h\log(d)}\nu_1^{-h\log(d)}2^{1-h\log(d)} + 2(h\log(d))^{2h\log(d)}\nu_2^{-2h\log(d)}\right)^{1/(2\log(d))} \eqs.
\end{align*}
For $n \geq n^*$ we have $\nu_1 \log(d) \leq \nu_2^2$ and therefore we get
\begin{equation*}
 \EE \left[\|\frac{1}{n} \sum_{i=1}^{n} \Bern{Z_{i}}  \|_{\sigma,\infty}^h \right] \leq \left(\frac{2eh\log(d)}{\nu_1}\right)^{h/2} \eqs.
\end{equation*}

\end{proof}

\end{document}